\documentclass[a4paper,12pt]{article}
\usepackage{amsfonts}
\usepackage{pifont}
\usepackage{caption}
\usepackage{graphicx, subfig}
\usepackage{bm}
\usepackage{latexsym,amsmath,amssymb,cite,amsthm}
\usepackage{color,eucal,enumerate,mathrsfs}
\usepackage[normalem]{ulem}
\usepackage{amsmath}
\usepackage[pagewise]{lineno}
\usepackage[pagebackref=true, colorlinks, linkcolor=blue,  anchorcolor=blue,citecolor=blue]{hyperref}

\numberwithin{equation}{section}
\theoremstyle{plain}
\newtheorem{theorem}{Theorem}[section]
\newtheorem{corollary}[theorem]{Corollary}
\newtheorem{lemma}[theorem]{Lemma}
\newtheorem{proposition}[theorem]{Proposition}
\theoremstyle{definition}
\newtheorem{problem}[theorem]{Open Problem}
\newtheorem{example}[theorem]{Example}
\newtheorem{definition}[theorem]{Definition}
\theoremstyle{remark}

\newtheorem{remark}[theorem]{Remark}

\newcommand\bdf{\begin{definition}}
\newcommand\bpr{\begin{proposition}}
\newcommand\brk{\begin{remark}}
\newcommand\blm{\begin{lemma}}
\newcommand\bexe{\begin{exercise}}
\newcommand\bexa{\begin{example}}
\newcommand\beqn{\begin{eqnarray*}}
\newcommand\edf{\end{definition}}
\newcommand\epr{\end{proposition}}
\newcommand\erk{\end{remark}}
\newcommand\elm{\end{lemma}}
\newcommand\eexe{\end{exercise}}
\newcommand\eexa{\end{example}}
\newcommand\eeqn{\end{eqnarray*}}

\newcommand{\ds}{\mathsf{d}}

\newcommand{\lip}[1]{{\mathrm{lip}}({#1})}

\newcommand{\mm}{\mathfrak m}
\newcommand{\ms}{(X,\ds,\mm)}

\newcommand{\rcd}{{\rm RCD}(K, \infty)}

\newcommand{\De}{\mathrm{D}}

\newcommand{\R}{\mathbb{R}}

\newcommand{\pr}{\mathcal{P}}

\newcommand{\Lip}{\mathop{\rm Lip}\nolimits}

\renewcommand{\d }{{\mathrm d}}

\newcommand{\restr}[1]{\lower3pt\hbox{$|_{#1}$}}

\newcommand{\nchi}{{\raise.3ex\hbox{$\chi$}}}

\title{\Large{Barycenter curvature-dimension condition for extended metric measure spaces}
}

\begin{document}
\author{Bang-Xian Han\thanks{School of   Mathematics,  Shandong University,  Jinan, China.  Email: hanbx@sdu.edu.cn. }
\and Dengyu Liu
\thanks{School of Mathematical Sciences, University of Science and Technology of China, Hefei, China. Email:  yzldy@mail.ustc.edu.cn}
\and Zhuonan Zhu
\thanks{School of Mathematical Sciences, University of Science and Technology of China, Hefei, China.  Email: zhuonanzhu@mail.ustc.edu.cn}
}

\date{\today} 
\maketitle

\begin{abstract}
	In this survey, we introduce a new curvature-dimension condition for extended metric measure spaces,  called Barycenter-Curvature Dimension condition,  from the perspective of Wasserstein barycenter.

\end{abstract}


\tableofcontents

\section{Optimal transport and displacement convexity}
Optimal transport theory stems from a practical problem:
\begin{quote}
\emph{What is the most economical way to transport materials from one place to another,  with given supply and demand?}
\end{quote}
In  mathematics,  this  can be stated as a minimization problem:
$$
   T_c(\mu,\nu)=\min_{\pi\in\Pi(\mu, \nu)} \int_{X\times X} c(x,y) \, \d\pi(x, y)
     $$
where $c:X\times X\rightarrow [0,\infty]$ is a cost function and $\Pi(\mu, \nu)$ is composed of all probability measure $\pi$ on $X\times X$ whose marginals are $\mu$ and $\nu$. 

Let  $(X,\ds)$ be a metric space and consider the  optimal transport problem with the cost function $c=\ds^2$.  The Wasserstein distance $W_2$ is defined by $W_2^2(\mu,\nu):=T_{\ds^2}(\mu,\nu)$.  It is immediate to see that $W_2$ is an extended metric (we allow  $W_2(\mu, \nu)=\infty$) on the space of Borel probability measures $\mathcal{P}(X)$. To make it become a metric, we often restrict  the study on a subspace $$\mathcal{P}_2(X,\ds):=\left\{\mu\in\mathcal{P}(X): \int_X \ds^2(x,x_0)\,\d x<\infty ~\text{for some}~x_0\in X\right\}.$$The metric space $(\mathcal{P}_2(X,\ds), W_2)$ is called Wasserstein space.  The geometry theory of  Wasserstein space  has a wide range of applications in many fields, including differential equations, geometric and functional inequalities. We refer to \cite{AG-U} and \cite{V-O} for a comprehensive discussion about optimal transport and  Wasserstein space.  In this survey, we mainly focus on the connection between optimal transport and the geometry of metric measure spaces.

\medskip

In his celebrated paper \cite{mccann1997convexity}, McCann studied the convexity of functionals on $(\mathcal{P}_2(\mathbb{R}^n, |\cdot|),  W_2)$.  We say that a functional $\mathcal{F}$  is $K$-displacement convex (in the sense of McCann) for some $K\in \R$,  if for any $\mu_0,\mu_1 \in \mathcal{P}_2(\mathbb{R}^n, |\cdot|)$,  there exists  a geodesic $\{\mu_t\}$ in the Wasserstein space, connecting  $\mu_0$ and $\mu_1$, such that 
$$
\frac{K}{2}t(1-t)W_2^2(\mu_0,\mu_1)+\mathcal{F}(\mu_t)\leq (1-t)\mathcal{F}(\mu_0)+t\mathcal{F}(\mu_1).
$$ 

\begin{example}[Potential energy]
Let $f$ be a nonnegative  convex function on $\mathbb{R}^n$. Then $\mathcal{V}(\mu):=\int_{\mathbb{R}^n}f\d \mu$ is $0$-displacement convex, or displacement convex for simplicity.
\end{example}

There is a family of functionals, called internal energies, which plays  crucial role in the study of metric measure spaces. In physics, this kind of functionals are more known as entropies since they have a thermodynamic interpretation of internal energy of a gas. Precisely, the internal energy $\mathcal{H}^U_\mm$ is defined as
\begin{equation*}
	\mathcal{H}^U_\mm(\mu):=
	\left \{\begin{array}{ll}
		\int U(\rho)\,\d \mm ~~~&\text{if}~ \mu=\rho\,\mm\\
		+\infty &\text{otherwise},
	\end{array}\right.
\end{equation*} 
where $U:[0,\infty)\rightarrow [0,\infty]$ is lower semicontinuous, convex and with super-linear  growth. When $U(x)=x \ln x$, $\mathcal{H}^U_\mm$ is known as Boltzmann entropy,  and it is usually denoted by ${\rm Ent}_{\mm}$.  We refers the readers to \cite{AGS-G} and \cite{SantambrogioBook} for more  functionals and applications.

\begin{proposition}[Boltzmann entropy, Theorem 2.2 in \cite {mccann1997convexity}]
	In $(\mathbb{R}^n, |\cdot|,\mathcal{L}^n)$, the Boltzmann entropy is displacement convex:
$$
{\rm Ent}_\mm(\mu_t)\leq (1-t){\rm Ent}_\mm(\mu_0)+t{\rm Ent}_\mm(\mu_1)~~~t\in [0,1]
$$
for any geodesic $\{\mu_t\}$ in the Wasserstein space.
\end{proposition}

\section{Synthetic theory of Ricci curvature}\label{synthetic}
Since the seminal paper of Otto and Villani  \cite{OttoVillani00},  the link between displacement convexity and  geometry of    metric measure spaces has been deeply studied from the perspective of Ricci curvature.  On one hand, Cordero-Erausquin, McCann and Schmuckenschl{\"a}ger \cite{CMS01} showed that  the Boltzmann entropy is $K$-displacement convex in the Wasserstein space over a compact Riemannian manifold with Ricci curvature bounded from below by $K$.  On the other hand,  von Renesse and Sturm \cite{SVR-T} proved that  displacement convexity actually implies lower Ricci curvature bounds. Thus, in the setting of Riemannian manifolds, lower Ricci curvature bounds can be characterized by  the displacement convexity of certain functionals. This fact is one of   main motivation and starting points  of Lott--Sturm--Villani theory about non-smooth metric measure spaces.

\subsection*{Curvature-dimension condition}
A synthetic curvature-dimension condition for metric measure spaces,  called {\rm CD}  condition  today,   was introduced independently by Lott--Villani \cite{Lott-Villani09} and Sturm \cite{S-O1}.  They adopt the interpolation inequality of Cordero-Erausquin,  McCann and Schmuckenschl{\"a}ger \cite{CMS01}  as a definition.

\begin{definition}[${\rm CD}(K,\infty)$ condition]\label{def:CD infty}
	Let $K \in \R$.  We say that a  metric measure space  $(X,\ds,\mm)$ verifies ${\rm CD}(K,\infty)$ condition,   if for any  $\mu_{0},\mu_{1} \in \pr_{2}(X,\ds)$, there exists a geodesic $\mu_t$ such that for all $t\in(0,1)$:
	\begin{equation}\label{eq:defCD infty}
		{\rm Ent}_{\mm}(\mu_t)\leq(1-t){\rm Ent}_{\mm}(\mu_0)+t{\rm Ent}_{\mm}(\mu_1)-\frac{K}{2}(1-t)tW_2^2(\mu_0,\mu_1).
	\end{equation}
\end{definition}

	It is worth to mention that ${\rm CD}(K,\infty)$ is a dimensionless curvature condition. There  is a variant,  called ${\rm CD}(K,N)$ condition,  which has a dimensional parameter $N$.  We refer to  \cite{S-O2} and \cite{V-O} for the definition.

     \subsection*{Riemannian curvature-dimension condition}
     In \cite{Ohta09}, Ohta proved that some Finsler manifolds satisfy $\rm CD$ condition.  Later,  Ambrosio, Gigli and Savar{\' e}, in \cite{AGS-M}, introduced a Riemannian variant of the curvature-dimension condition, called Riemannian Curvature-Dimension condition ($\rm RCD$ for short),  ruling out Finsler manifolds.
     
     Recall that Cheeger energy (cf. \cite{C-D})  is a functional defined on $ L^2(X, \mm)$:
     $$
     {\rm Ch}(f):=\frac{1}{2}\inf\left\{\liminf_{i\rightarrow\infty}\int_X \left|\lip {f_i}\right|^2\,\d\mm:f_i\in {\rm Lip}(X, \ds),\left\|f_i-f\right\|_{L^2}\rightarrow 0\right\},
     $$
     where
     $$
     \left|\lip h\right|(x):=\limsup_{y\rightarrow x}\frac{\left|h(y)-h(x)\right|}{\ds(x,y)}
     $$
     for $h\in \Lip(X,\ds)$. We define the Sobolev space  $W^{1,2}\ms$ by
     $$
     W^{1,2}\ms:=\left\{f\in L^2(X, \mm):  {\rm Ch}(f)<\infty\right\}.
     $$
     Note that, a smooth Finsler manifold  is Riemannian if and only if  its Sobolev space $W^{1,2}$ is a Hilbert space; see  \cite[Section 4]{gigli2023giorgi}. This fact motivates the following definition, proposed by Gigli in \cite{G-O}.
     \begin{definition}[Infinitesimally Hilbertian space]
     	We say that a metric measure space  $(X,\ds,\mm)$ is an infinitesimally Hilbertian space if $W^{1,2}\ms$ is  Hilbert.  Equivalently,   the Cheeger energy ${\rm Ch}$ is a quadratic form  in the sense that
     	\begin{equation*}
     		{\rm Ch}(f+g)+{\rm Ch}(f-g)=2{\rm Ch}(f)+2{\rm Ch}(g)\qquad \text{for all} ~f,g\in W^{1,2}\ms.
     	\end{equation*}
     \end{definition}
     
     We can then give the following definition:
     
     \begin{definition}[${\rm RCD}(K,\infty)$ condition, see \cite{AGS-M} and \cite{AGMR-R}]
     	Let $K \in \R$. A metric measure space  $(X,\ds,\mm)$ verifies Riemannian curvature bounded from below by $K$ (or $(X,\ds,\mm)$ is an ${\rm RCD}(K,\infty)$ space) if it satisfies ${\rm CD}(K,\infty)$ condition and $W^{1,2}\ms$ is a Hilbert space.
     \end{definition}

\subsection*{Evolution Variation Inequality (EVI)}

In \cite{daneri2008eulerian}, Daneri and Savar{\'e}  provided an alternative way to study  synthetic curvature-dimension condition in possibly infinite dimensional,  non-compact spaces (see Ambrosio, Erbar and Savar\'e's paper \cite{ambrosio2016optimal} for more discussion in the setting of extended metric measure spaces).  They established the existence of a stronger formulation of gradient flows, called ${\rm EVI}_K$ (Evolution Variation Inequality) gradient flows,  on Riemannian manifolds.

\begin{definition}[See \cite{ambrosio2016optimal}]\label{def:EVI}
	Let $(Y, \ds_Y)$ be an extended metric space and  $E:Y\rightarrow \mathbb{R}\cup \left\{+\infty\right\}$ be a lower semi-continuous functional. Denote ${\De (E)}:={\left\{E<+\infty\right\}}$. For any $y_0\in \De(E)$,  we say that $(0,\infty)\ni t\rightarrow y_t \in Y$ is an ${\rm EVI}_{K}$-gradient flow  of $E$, starting from $y_0$,  if it is a locally absolutely continuous curve,  such that $y_t\overset{\ds_Y}{\rightarrow}y_0$ as $t\rightarrow 0$ and the following inequality is satisfied: for any $z\in \De (E)$ satisfying $\ds_Y(z, y_t)<\infty$ for some (and then all) $t\in (0, \infty)$:
	\begin{equation}
		\frac{1}{2}\frac{\d }{\d t}\ds_Y^2(y_t,z)+\frac{K}{2}\ds_Y^2(y_t,z)\leq	E(z)-E(y_t),\qquad \text{for a.e. } t\in(0,\infty).
	\end{equation}
\end{definition}

It has been  proved in \cite{daneri2008eulerian} that this stronger formulation of gradient flows   implies (geodesic) convexity.  As an application,   one can prove McCann's displacement convexity by showing the existence of  ${\rm EVI}_K$-gradient flows in Wasserstein space.

The ${\rm EVI}_K$-gradient flow  encodes information not only about the functional $E$ itself but also the structure of the underling space.  On one hand,  Ohta and Sturm \cite{OS-H} proved that the heat flow on a Finsler manifold is linear if and only if it is Riemannian. On the other hand,  by Daneri and Savar{\'e} \cite{daneri2008eulerian},   we know the existence of ${\rm EVI}_K$-gradient flows of the relative entropy implies the linearity of the heat flow.  This suggests that ${\rm EVI}_K$-gradient flow is closely related to the Riemannian (infinitesimally Hilbertian) structure  of the underlying space. Remarkably,    the $\rm RCD$ condition can be equivalently defined via heat flow.
\begin{theorem}[\cite{AGS-M}, Theorem 5.1; \cite{AGMR-R}, Theorem 6.1]\label{rcdevi}
	Let $(X,\ds,\mm)$ be a metric measure space and $K\in\mathbb{R}$. Then $(X,\ds,\mm)$ satisfies ${\rm RCD}(K,\infty)$ condition if only and if  $(X,\ds,\mm)$ is a length space satisfying an exponential growth condition
	$$
	\int_X e^{-c\ds(x_0,x)^2}\,\d\mm(x)<\infty,\qquad \text{for some } x_0 \in X \text{ and } c>0,
	$$
	and for all $\mu\in\mathcal{P}_2(X, \ds)$ there exists an ${\rm EVI}_K$-gradient flow of ${\rm Ent}_{\mm}$, in $(\mathcal{P}_2(X, \ds),W_2)$,  starting from $\mu$.
\end{theorem}

\begin{remark}
There is another well known synthetic curvature-dimension condition, called Bakry-{\' E}mery condition ({\rm BE} condition for short), named after Bakry and {\' E}mery \cite{BEdiffusions} (see \cite{LedouxICM, BakryGentilLedoux14} for more discussion on this topic). In \cite{AGS-B},  Ambrosio, Gigli and Savar{\'e} proved that BE condition,   under some natural hypothesis,  implies the existence of ${\rm EVI}_K$-gradient flows,  which then implies 
${\rm RCD}$ condition. Conversely, in \cite{AGS-M}, they show that $\rm BE$ condition could be deduced from {\rm RCD} condition.  
\end{remark}

	\begin{example}
		The existence of the  ${\rm EVI}_K$-gradient flow of a $K$-convex function $E$,  is valid on Euclidean spaces, Hilbert spaces, Alexandrov spaces, ${\rm CAT}$ spaces; see \cite{AGS-G} and \cite[Section 3.4]{muratori2020gradient}. 
	\end{example}

	\begin{example}
		For the Wasserstien space $(\mathcal{P}_2(X,\ds),W_2)$ over metric measure spaces, the existence of the  ${\rm EVI}_K$-gradient flow is valid,  when $X$ is an Euclidean space, a smooth Riemannian manifold with uniform Ricci lower bound, an Alexandrov space with  curvature bounded from below or an $\rm RCD$ space; see  also \cite{AGS-M} and \cite{GKO-H}.
	\end{example}
	
	\begin{example}
		The existence of ${\rm EVI}_K$-gradient flows also holds on the Wasserstien space $(\mathcal{P}(X),W_2)$ (we denote the Wasserstein space over an extended metric space  by $(\mathcal{P}(X),W_2)$ since $\mathcal{P}_2(X, \d)$ makes no sense) over extended metric measure spaces, such as a Wiener space \cite{FSS10} and configuration spaces \cite{EH-Configure},  see \cite[Section 13]{ambrosio2016optimal} for more discussion. We will discuss its connection with Wasserstein barycenter  problem in Section \ref{barycenter}.
	\end{example}
\section{Wasserstein barycenter}\label{barycenter}
In this section we review some recent results obtained by the authors in \cite{HLZ-BC1} concerning Wasserstein barycenter problem.

We say that  $(Y,\ds_Y)$ is an extended metric space if $\ds_Y : Y \times Y \to  [0, +\infty]$ is a symmetric function satisfying the triangle inequality, with $\ds_Y(x, y) = 0$ if and only if $x = y$.
Let $\Omega$  be a Borel probability measure  on $(Y,\ds_Y)$.  A   minimizer of the function $$x\mapsto\int_{Y} \ds_Y^2(x,y)\,\d \Omega(y)$$ is called a barycenter of $\Omega$,  and the variance of $\Omega$ is defined as
\[
{\rm Var}(\Omega):=\inf_{x\in X} \int_{Y} \ds_Y^2(x,y)\,\d \Omega(y).
\]  

In the Wasserstein space $(\mathcal{P}_2(X,\ds),W_2)$, a barycenter of $\Omega$ is also called Wasserstein barycenter.  This kind of problems, called Wasserstein barycenter problems,  draw particular interests in optimal transport and its related fields, as it gives a natural  way to  average several measures. We remark that the barycenter problem is independent of the linearity of the underling space,  and is essentially a non-linear problem.

In  Euclidean space,  Agueh--Carlier \cite{AguehCarlier}  established the well-posedness (i.e. existence, uniqueness and the absolute continuity) of the Wasserstein barycenter problem and prove the following Jensen's inequality.
\begin{theorem}[\cite{AguehCarlier}]
	In $(\mathbb{R}^n, |\cdot|,\mathcal{L}^n)$, the Boltzmann entropy satisfies the Wasserstein Jensen's inequality:
	\begin{equation}
		{\rm Ent}_{\mathcal{L}^n}(\bar{\mu})\leq \int_{\mathcal{P}_2(\mathbb{R}^n, |\cdot|)} {\rm Ent}_{\mathcal{L}^n}(\mu)\,\d\Omega(\mu)
	\end{equation}
	where $\Omega$ is supported in finitely many measures over $\mathcal{P}_2(\mathbb{R}^n,|\cdot|)$ with finite variance and $\bar{\mu}$ is a  barycenter of $\Omega$.
\end{theorem}

Moreover,  Kim--Pass \cite{KimPassAIM} extended the above  results to compact Riemannian manifolds;  Ma \cite{ma2023} improves Kim--Pass's results in  Riemannian manifolds when $\Omega$ is general distributions (with infinite support); Jiang \cite{JiangWasserstein}  proved the well-posedness in Alexandrov spaces. {\bf As a consequence}, the following generalized Wasserstein Jensen's inequality has been proved, which characterizes the convexity of certain functionals via Wasserstein barycenter.

Let  $\mathcal F: \pr_2(X, \ds) \to \R \cup \{+\infty\}$  be a $K$-displacement convex functional  in the sense of McCann \cite{mccann1997convexity},     $\Omega$ be a probability measure on $\mathcal{P}_2(X, \ds)$. We say that  Wasserstein Jensen's inequality holds if there is a  barycenter $\bar{\mu}$ of $\Omega$ such that
\begin{equation}\label{jensen}
	\mathcal F(\bar{\mu})\leq \int_{\mathcal{P}_2(X, \ds)}\mathcal F(\mu)\,\d\Omega(\mu)-\frac{K}{2}\int_{\mathcal{P}_2(X, \ds)}W_2^2(\bar{\mu},\mu)\,\d\Omega(\mu).\tag{WJI}
\end{equation}

The above inequality can be understood as a generalization of  ``displacement convexity'',  from the perspective of Wasserstein barycenter. Motivated by the synthetic theory of curvature-dimension condition, it is natural to ask:
\begin{quote}
 \emph{Does the Wasserstein Jensen's inequality hold on $\rm RCD$ and $\rm CD$ spaces?}
 \end{quote}

One may try to mimic the argument  used for $\mathbb{R}^n$ or Riemannian manifolds: firstly prove the  regularity of  Wasserstein barycenter, then prove the Jensen's  inequality using the  first order balance condition.   A huge barrier lies in the lack of sectional curvature bound, which plays a crucial role in deriving the regularity of Wasserstein barycenter. A new feature in our work is that we deal with  Wasserstein barycenter  in a synthetic approach, by treating Jensen's inequality as an ``{\bf a priori estimate}", so that we do not need to prove the regularity of Wasserstein barycenter at first.  In fact,   the Wasserstein Jensen's inequality  already encodes  the regularity of barycenter. To see this,  applying $\mathcal{F}={\rm Ent}_{\mm}$ in \eqref{jensen}, we get
\begin{equation}
	{\rm Ent}_{\mm}(\bar{\mu})\leq \int_{\mathcal{P}_2(X, \ds)}{\rm Ent}_{\mm}(\mu)\,\d\Omega(\mu)-\frac{K}{2}\int_{\mathcal{P}_2(X, \ds)}W_2^2(\bar{\mu},\mu)\,\d\Omega(\mu).
\end{equation}
So  $	{\rm Ent}_{\mm}(\bar{\mu})<+\infty$ in  case  that $\int_{\mathcal{P}_2(X, \ds)}{\rm Ent}_{\mm}(\mu)\,\d\Omega(\mu)<+\infty$ and $\Omega$ has finite variance.

\medskip

Motivated by Daneri and Savar{\'e} \cite{daneri2008eulerian}, we realize that the existence of ${\rm EVI}_K$ gradient flow implies the Jensen's inequality.

\begin{theorem}[${\rm EVI}_K$ implies Jensen's inequality]\label{EVI JI}
	Let $(Y,\ds_Y)$ be an extended metric space,  $K\in \R$,  $E$ be as in Definition \ref{def:EVI} and $\mu$ be a  probability measure  over $Y$ with finite variance.  Let $\epsilon\geq 0.$ Assume that ${\rm EVI}_K$ gradient flow of $E$ exists for some initial data $y\in Y$  satisfying
	\begin{equation}\label{nearly barycenter}
		\int_Y \ds_Y^2(y,z)\,\d\mu(z)\leq {\rm Var}(\mu)+\epsilon.
	\end{equation}
	Then the following inequality holds:
	
	\begin{equation}\label{NearlyJI}
		E(y_t)\leq \int_{Y} E(z)\,\d\mu(z)-\frac{K}{2}{\rm Var}(\mu)+\frac{\epsilon}{2I_K (t)},
	\end{equation}
	where $\{y_t\}$ is the ${\rm EVI}_K$ gradient flow of E  starting from $y$ and $I_K(t)$:=$\int_{0}^{t}e^{Kr}\,\d r$. In particular, if   $\mu$ has a barycenter $\bar{y}$ and there is  an ${\rm EVI}_K$ gradient flow of $E$ starting from $\bar y$, then  we have the Jensen's inequality:
	\begin{equation}\label{JI}
		E(\bar{y})\leq \int_{Y} E(z)\,\d\mu(z)-\frac{K}{2}\int_{Y}\ds^2(\bar{y},z)\,\d\mu(z).
	\end{equation}
\end{theorem}
The following proposition,   proved in \cite[Proposition 3.1]{daneri2008eulerian} (see also \cite{ambrosio2016optimal}),  provides an integral version of ${\rm EVI}_K$ type gradient flows. We will use this formulation to prove Jensen's inequality.
\begin{proposition}[Integral version of ${\rm EVI}_K$]\label{integral version}
	Let $E$, $K$ and $\{y_t\}_{t>0}$ be as in Definition \ref{def:EVI}, then $\{y_t\}_{t>0}$ is an ${\rm EVI}_{K}$ gradient flow if and only if it satisfies 
	\begin{equation}\label{eq:integral version}
		\frac{e^{K(t-s)}}{2}\ds_Y^2(y_t,z)-\frac{1}{2}\ds_Y^2(y_s,z)\leq I_K(t-s)\big(E(z)-E(y_t)\big),\quad \forall ~ 0\leq s\leq t,
	\end{equation}
	where $I_K(t)$:=$\int_{0}^{t}e^{Kr}\,\d r$.
\end{proposition}

\bigskip

\begin{proof}[Proof of Theorem \ref{EVI JI}]
	Integrating \eqref{eq:integral version} in $z$ with respect to  $\mu$ and choosing $s=0$,  we get
	\begin{equation}\label{EVI JI1}
		\frac{e^{Kt}}{2}\int_{Y}\ds_Y^2(y_t,z)\,\d\mu(z)-\frac{1}{2}\int_{Y}\ds_Y^2(y,z)\,\d\mu(z)\leq I_K(t)\left(\int_{Y} E(z)\,\d\mu(z)-E(y_t)\right).
	\end{equation}
From the definition of barycenter, we can see that  $\int_{Y}\ds_Y^2(y_t,z)\,\d\mu(z)\geq{\rm Var}(\mu)$. Combining with  \eqref{nearly barycenter} and \eqref{EVI JI1}, we get
	\begin{equation}\label{EVI JI2}
		\frac{e^{Kt}-1}{2}{\rm Var}(\mu)-\frac{\epsilon}{2}\leq I_K(t)\left(\int_{Y} E(z)\,\d\mu(z)-E(y_t)\right).
	\end{equation} 
	Dividing both sides of \eqref{EVI JI2} by $I_K(t)$, we get \eqref{NearlyJI}. In particular, if $y$ is a barycenter of $\mu$,  we can take $\epsilon=0$ in \eqref{NearlyJI}, so that
	$$
	E(y_t)\leq \int_{Y} E(z)\,\d\mu(z)-\frac{K}{2}{\rm Var}(\mu).
	$$
	Letting $t\rightarrow 0$, and recalling the lower semi-continuity of $E$ and $y_t\overset{\ds_Y}{\rightarrow} y_0$, we get Jensen's inequality \eqref{JI}.
\end{proof}
Combining Theorem \ref{rcdevi} and  Theorem \ref{EVI JI} we  establish  Wasserstein Jensen's inequality in   $\rm RCD$ spaces.

\begin{corollary}\label{coro:rcd}Let  $\ms$ be an $\rcd$ space.  Let  $\Omega$ be a  probability measure over $\mathcal{P}_2(X, \ds)$ with finite variance.  
	Then  for any barycenter $\bar{\mu}$ of $\Omega$,  it holds
	\begin{equation}
		{\rm Ent}_{\mm}(\bar{\mu})\leq \int_{\mathcal{P}_2(X, \ds)} {\rm Ent}_{\mm}(\mu)\,\d\Omega(\mu)-\frac{K}{2}\int_{\mathcal{P}_2(X, \ds)} W_2^2(\bar{\mu},\mu)\,\d\Omega(\mu).
	\end{equation}
\end{corollary}

Furthermore, on some extended metric measure spaces, such as abstract Wiener spaces \cite{ambrosio2016optimal} and configuration spaces \cite{EH-Configure},  ${\rm EVI}_K$ gradient flow of the relative entropy is  well-posed in the sense of Definition \ref{def:EVI}. To apply Theorem \ref{EVI JI} for extended metric measure spaces, one should take care of  the existence of the gradient flow from a point  with finite distance to the domain of the relative entropy. In general,  this is a highly  non-trivial problem.

\begin{corollary}[See \cite{FU} and \cite{FSS10}]\label{coro:wn}Let  $(X,H,\gamma)$ be an abstract Wiener space equipped with the  Cameron--Martin distance  $\ds_H$ and the Gaussian measure $\gamma$.  Let  $\Omega$ be a  probability measure over $\mathcal{P}(X)$ with finite variance.  
	
	Then  for any barycenter $\bar{\mu}$ of $\Omega$ it holds
	\begin{equation}\label{eq:wn}
		{\rm Ent}_{\gamma}(\bar{\mu})\leq \int_{\mathcal{P}(X)} {\rm Ent}_{\gamma}(\mu)\,\d\Omega(\mu)-\frac{1}{2}\int_{\mathcal{P}(X)} W_2^2(\bar{\mu},\mu)\,\d\Omega(\mu).
	\end{equation}
	
	Furthermore, if  $\int_{\mathcal{P}(X)} {\rm Ent}_{\gamma}(\mu)\,\d\Omega(\mu)<+\infty$,  $\Omega$ has a unique barycenter.
\end{corollary}
\begin{corollary}[See \cite{EH-Configure}]
	Let  $(\mathbf{Y},\ds_{\mathbf Y},\pi)$ be an extended metric measure space where  $\mathbf{Y}$ is the configuration space  over a Riemannian manifold  $M$,    $\ds_{\mathbf Y}$ is the  intrinsic metric  and $\pi$  is the Poisson  measure .   Assume that $M$ has Ricci curvature bounded from below by $K \in \R$.

	Then  for any   barycenter $\bar{\mu}$ of a probability measure  $\Omega$ over $\mathcal{P}(\mathbf{Y})$ with finite variance,  its holds
	\begin{equation*}
		{\rm Ent}_{\pi}(\bar{\mu})\leq \int_{\mathcal{P}(\mathbf{Y})} {\rm Ent}_{\pi}(\mu)\,\d\Omega(\mu)-\frac{K}{2}\int_{\mathcal{P}(\mathbf{Y})} W_2^2(\bar{\mu},\mu)\,\d\Omega(\mu).
	\end{equation*}
	
	Furthermore,  if $\int_{\mathcal{P}(\mathbf{Y})} {\rm Ent}_{\pi}(\mu)\,\d\Omega(\mu)<+\infty$,   $\Omega$ has a unique barycenter.
\end{corollary}

\brk
Concerning general $\rm CD$ spaces,  it seems hard to prove  Wasserstein Jensen's inequality.  At the moment, we even do not know whether this is true for Finsler manifolds or not.
\erk
\section{Barycenter Curvature-Dimension condition}

In the previous sections,  we have seen the connection between displacement convexity, Wasserstein barycenter, Jensen's inequality and curvature-dimension condition of  metric measure spaces. This motivates us to propose a new curvature-dimension condition, which covers some non-trivial extended metric measure spaces, e.g. Wiener spaces and Configuration spaces over a Riemannian manifold.

Due to the special role of barycenter, we call this new curvature-dimension condition {{B}arycenter {C}urvature-{D}imension condition} and denote it by {\rm BCD} for short.

\begin{definition}[${\rm BCD}(K,\infty)$  condition]\label{def:bcd}
	
	Let $K \in \R$.  We say that an {extended metric measure  space} $(X,\ds,\mm)$ verifies ${\rm BCD}(K,\infty)$ condition,   if for {any}  probability measure $\Omega\in \pr_2(\pr(X), W_2)$,  concentrated on finitely many measures, there {exists} a barycenter $\bar{\mu}$ of  $\Omega$ such that the following Jensen's  inequality holds:
	\begin{equation}\label{eq:defBCD infty}
		{\rm Ent}_{\mm}(\bar{\mu})\leq\int_{\mathcal{P}(X)}{\rm Ent}_{\mm}(\mu)\,\d\Omega(\mu)-\frac{K}{2}{\rm Var}(\Omega).
	\end{equation}

\end{definition}
\begin{remark}
	If we take $\Omega=(1-t)\delta_{\mu_0}+t\delta_{\mu_1}$ and $(X, \ds)$ is a geodesic space,  BCD condition implies the Lott--Sturm--Villani's curvature-dimension condition. 
\end{remark}
\brk
Similar to ${\rm CD}(K, N)$ condition (cf. \cite{Lott-Villani09, S-O2}) and  ${\rm RCD}(K, N)$ condition (cf. \cite{EKS-O}), there is  a  finite dimensional version of $\rm BCD$ condition,  called ${\rm BCD}(K, N)$ condition; see  \cite[Section 6]{HLZ-BC1} for details.
\erk

Not surprisingly, we can prove that the family of compact  metric measure spaces satisfying barycenter curvature-dimension condition,  is closed in measured Gromov--Hausdorff topology. 

\begin{theorem}[Stability in measured Gromov--Hausdorff topology]
	Let $\left\{(X_i, \ds_i, \nu_i)\right\}_{i=1}^{\infty}$ be a sequence of compact ${\rm BCD}(K,\infty)$ metric measure spaces with $K\in\mathbb{R}$. If  $(X_i, \ds_i, \nu_i)$ converges to $(X, \ds, \nu)$ in the measured Gromov--Hausdorff sense as $n\rightarrow \infty$, then $(X, \ds, \nu)$ is also a ${\rm BCD}(K,\infty)$ space. 
\end{theorem}

\bigskip

Next we introduce  some new geometric inequalities,  as simple but interesting applications of our BCD theory. For simplicity, we  will only deal with metric measure spaces.  We  refer the readers to a  recent paper \cite{KW-BS} by Kolesnikov and Werner  for more applications of Wasserstein barycenter.  In addition,  see \cite{HLZ-BC1} for more applications of our BCD theory to multi-marginal optimal transport.

\begin{proposition}[Multi-marginal logarithmic Brunn--Minkowski inequality]
	Let $\ms$ be a ${\rm BCD}(0,\infty)$ metric measure space,  $E_1,..., E_n$ be  bounded measurable sets with positive measure  and $\lambda_1,\dots, \lambda_n\in (0, 1)$ with $\sum_i \lambda_i=1$. Then
	\[
	\mm(E) \geq \mm(E_1)^{\lambda_1}... \mm(E_n)^{\lambda_n}
	\]
	where 
	$$
	E:=\left\{x ~\text{is a barycenter of}~\sum_{i=1}^n \lambda_i\delta_{x_i}: x_i\in E_i, i=1,\dots, n \right\}.
	$$
\end{proposition}

\medskip

\begin{proposition}[A     functional Blaschke--Santal\'o type  inequality]
	Let $\ms$ be  a ${\rm BCD}(1,\infty)$ metric measure space. Then we  have
	\begin{equation*}\label{eq1:th3}
		\prod_{i=1}^k \int_X e^{f_i}\,\d \mm\leq 1 
	\end{equation*}
	for any  measurable functions   $f_i$ on $X$ such that $\frac {e^{f_i}}{\int e^{f_i}\,\d \mm}\in \pr_2(X, \ds)$ and
	\begin{equation*}\label{eq:duality}
		\sum_{i=1}^k f_i(x_i)\leq  \frac 1{2} \inf_{x\in X} \sum_{i=1}^k  \d(x, x_i)^2\qquad \forall x_i\in X, i=1,2,...,k.
	\end{equation*}
\end{proposition}

\bigskip

At the end of this paper, we list some  problems we would like to study in the future.

\begin{problem} 
As explained in Section \ref{synthetic}, the difference between $\rm CD$ condition and $\rm RCD$ condition lies in whether they have a ``Riemannian structure" or not. It is natural to ask: Do Finsler manifolds with lower   Ricci curvature   bound from below satisfy BCD condition? Is there a ``Riemmanian Barycenter-curvature dimension condition"? 
\end{problem}

\begin{problem}Can we find more similarities and differences between  $\rm CD$ condition and  $\rm BCD$ condition?  
\end{problem}

\begin{problem} In the preceding examples of BCD spaces, they are either geodesic or ``almost geodesic". It is interesting to find a $\rm BCD$ space with less geodesics.  Moreover,   is there  a  $\rm BCD$ space which has a fractal structure?
\end{problem}

\bigskip
\noindent \textbf{Declaration.}
{The  authors declare that there is no conflict of interest and the manuscript has no associated data.}

\medskip

\noindent \textbf{Acknowledgement}:   This work is supported in part by Young Scientist Programs of the Ministry of Science and Technology of China (2021YFA1000900, 2021YFA1002200), and National Natural Science Foundation of China  (12201596).

\addcontentsline{toc}{section}{References}
\def\cprime{$'$}
\providecommand{\bysame}{\leavevmode\hbox to3em{\hrulefill}\thinspace}
\providecommand{\MR}{\relax\ifhmode\unskip\space\fi MR }
\providecommand{\MRhref}[2]{%
  \href{http://www.ams.org/mathscinet-getitem?mr=#1}{#2}
}
\providecommand{\href}[2]{#2}

\end{document}